\documentclass[12pt]{amsart}

\usepackage[utf8]{inputenc}
\usepackage{hyperref}
\usepackage{comment}
\usepackage{amssymb,amsthm,amsmath,mathrsfs,amssymb}
\usepackage[mathscr]{euscript}

\usepackage{enumitem}

\makeatletter
\def\paragraph{\@startsection{paragraph}{4}%
	\z@\z@{-\fontdimen2\font}%
	{\normalfont\bfseries}}
\makeatother

\usepackage{xcolor}

\newcommand{\C}{\mathbb{C}}

\newcommand{\R}{\mathbb{R}}

\newcommand*\colvec[3][]{
	\begin{pmatrix}\ifx\relax#1\relax\else#1\\\fi#2\\#3\end{pmatrix}
}
\newcommand*\abs[1]{
	\left|#1\right|
}

\newcommand*\deriv{
	\mathrm{d}
}
\newcommand*\Leb{
	\mathrm{Leb}
}

\newcommand\restr[2]{{% we make the whole thing an ordinary symbol
		\left.\kern-\nulldelimiterspace % automatically resize the bar with \right
		#1 % the function
		\littletaller % pretend it's a little taller at normal size
		\right|_{#2} % this is the delimiter
	}}
	\newcommand{\littletaller}{\mathchoice{\vphantom{\big|}}{}{}{}}

\newcommand{\norm}[1]{\left\lVert#1\right\rVert}
\newtheorem{thm}{Theorem}
\newtheorem*{thm*}{Theorem}
\newtheorem{lem}{Lemma}[section]

\theoremstyle{definition}

\newtheorem*{rem}{Remark}
\newtheorem*{rems}{Remarks}
\newtheorem*{assumption}{Assumption}
\newtheorem{defn}{Definition}
\newtheorem{notation}{Notation}
\newtheorem{claim}{Claim}
\renewcommand{\leq}{\leqslant}
\renewcommand{\geq}{\geqslant}
\normalsize \linespacing=\baselineskip

\evensidemargin=.65in{}
\usepackage{fullpage}
\title{Quantitative Fourier decay for Patterson-Sullivan measures of dimension larger than $1/2$}

\author{F\'{e}lix Lequen}
\address{Université Sorbonne Paris Nord, Villetaneuse, France}
\email{lequen@math.univ-paris13.fr}

\author{Tuomas Sahlsten}
\address{University of Helsinki, Helsinki, Finland}
\email{tuomas.sahlsten@helsinki.fi}

\thanks{F.L. was supported by T. Orponen’s grant from the Research Council of Finland via the project Approximate Incidence
Geometry, grant no. 35545, and from ANR-SNF project \emph{``Equidistribution in Number Theory''} (FNS no. 10.003.145 and ANR-24-CE93-0016). T.S. was supported by the Research Council of Finland's grant \emph{``Quantum chaos of large and many body systems''} (Nos. 347365, 353738).}

\begin{document}

\begin{abstract}
We give an elementary proof of power Fourier decay for Patterson-Sullivan measures associated to convex co-compact Schottky groups of dimension $\delta>\frac12$, obtaining the explicit decay exponent $\frac{\delta(2\delta - 1)}{(2\delta + 1)(3 - \delta)}$. The proof replaces the technical machinery used in previous works, such as sum-product estimates, renewal theory, Dolgopyat methods, and $L^2$ flattening, by elementary oscillatory integral estimates, hyperbolic geometry, and a duality argument based on the transpose Schottky group.
\end{abstract}

\maketitle

\section{Introduction}

Let $\Gamma$ be a convex co-compact Schottky subgroup of ${\rm PSL}_2(\mathbb{R})$. Then $\Gamma$ acts on 
$\partial \mathbb{H}^2=\R\cup\{\infty\}$ by M\"obius transformations.
The \emph{limit set} $\Lambda_\Gamma$ is the minimal non-empty closed $\Gamma$-invariant subset of $\R\cup\{\infty\}$. It is a compact Cantor set of Hausdorff dimension $\delta\in(0,1)$ with a Schottky coding by a subshift of finite type. The associated \textit{Patterson-Sullivan measure} $\mu$ (PS measure for short) is the unique $\delta$-conformal probability measure supported on $\Lambda_\Gamma$, characterised by
$
d\gamma_\ast\mu(x) = \abs{\gamma'(x)}^\delta d\mu(x)
$ for all $\gamma\in\Gamma$. In their groundbreaking work, Bourgain and Dyatlov \cite{BD} proved that there is $\eta > 0$ such that the Fourier transform $
\widehat{\mu}(\xi) = \int e^{-2\pi i\xi x}\,d\mu(x)$, $\xi\in\R,$
satisfies $|\widehat{\mu}(\xi)| \lesssim |\xi|^{-\eta}$ for all $|\xi| \geq 1$, using the sum-product theorem of Bourgain \cite{Bourgain}. Their motivation was to gain an improvement of the essential spectral gap under pressure condition depending only on $\delta$, via a Fractal Uncertainty Principle \cite{DyatlovZahl}. See also a renewal theoretic approach by Li \cite{LiStationary}. The result was also generalised to three dimensions in \cite{LNP} and in \cite{BKS} a new proof based on $L^2$ flattening that applied in any dimension.

More generally, Fourier decay for dynamically defined measures has attracted considerable attention in recent years (see e.g. \cite{AlgomChangWuWu,AlgomHertzWangLog,AlgomHertzWangPointwise,AlgomHertzWang,AlgomHertzWangPlane,Kaufman,Tsujii,SahlstenStevens,BakerSahlsten,BakerBanaji,LiSahlstenAffine,LiENS,LiSahlsten,LindenstraussVarju,VarjuYu,Rapaport,LeclercJulia,LeclercBunched,LeclercHyperbolic,LeclercOscillatory}) due to links to various related areas including Diophantine approximation \cite{PollingtonVelaniZafeiropoulosZorin}, multiscale additive combinatorics \cite{BD,SahlstenStevens,BakerSahlsten,LeclercPaukkonenSahlsten}, exponential mixing of Anosov systems \cite{LeclercFibonacci} and quantum chaos \cite{BD}, for which we refer e.g. to the survey \cite{Sahlsten}. Despite all this progress, obtaining explicit quantitative Fourier decay exponents remains difficult. There has been work in some self-similar cases \cite{Lee,BakerBanajiSlow}, as well as their pushforwards by non-linear maps \cite{BanajiYu}. 

Here we deal with the so-called ``totally non-linear’’ case of \cite{SahlstenStevens}, where existing proofs rely on sophisticated tools such as sum-product theory \cite{BD,Bourgain}, renewal theory \cite{LiStationary,AlgomHertzWang}, or $L^2$ flattening \cite{BKS,Khalil}. In this paper we show that for Patterson-Sullivan measures of dimension $\delta > 1/2$  none of this machinery is required. Instead, Fourier decay follows from a combination of oscillatory integral estimates, hyperbolic geometry and a non-concentration argument for the transpose Schottky group. Crucially, this allows us to gain an explicit polynomial decay exponent.

\begin{thm}\label{thm:main}
Suppose $\delta > 1/2$. There exists a constant $C = C(\Gamma) > 0$ such that
\[
\abs{\widehat{\mu}(\xi)}\le C\,\abs{\xi}^{-\frac{\delta(2\delta - 1)}{(2\delta + 1)(3 - \delta)}}\qquad\text{for all }\abs{\xi}\ge1.
\]
\end{thm}

Our method extends to the Patterson-Sullivan setting the oscillatory integral approach introduced by Kaufman \cite{Kaufman} and subsequently developed by Queff\'elec-Ramar\'e \cite{QueffelecRamare} and by Jordan and the second author \cite{JS} for Gibbs measures of the Gauss map. For this a key ingredient is a ``non-concentration estimate" which appears in the work of Bourgain-Dyatlov \cite{BD}. In a different guise, for the Gauss map, this is a form of the so-called uniform non-integrability condition as it appears in notes of Naud \cite{Naud}. We can think of it as quantifying the non-linearity of the system. Moreover, we rewrite this argument in terms of the transpose of the original Schottky group in the spirit of the ``dual IFS" \cite{BKT}.

\begin{rems} 
\begin{itemize}
\item[(1)] The restriction $\delta>1/2$ arises from our approximation procedure and oscillatory integrals at scale $\abs{\xi}^{-1/2}$. The method does not seem to work below $1/2$.

\item[(2)] Our decay exponent $\frac{\delta(2\delta-1)}{(2\delta+1)(3-\delta)}$ differs slightly from the exponent $\frac{\delta(2\delta-1)}{(2\delta+1)(4-\delta)}$ obtained by Queff\'elec-Ramar\'e \cite{QueffelecRamare}, solely through our sharper approximation lemma.

\item[(3)] In \cite{DyatlovZahl,BD}, it is shown Fourier decay $|\widehat{\mu}(\xi)|\lesssim |\xi|^{-\alpha}$ for PS measures implies an improvement for the essential spectral gap $\frac12-\delta$ of the Laplacian by $\frac{\alpha}{4}$. Thus optimal decay (i.e. PS measures are Salem) would only yield the gap $\frac12-\delta+\frac{\delta}{8}$ for $\delta \in [\frac{1}{2},\frac{1}{2} + \frac{1}{14}]$ \cite{BD}, whereas the Jakobson-Naud conjecture \cite[Conjecture 2]{JakobsonNaud} predicts the larger gap $\frac12-\delta+\frac{\delta}{2}$. Our exponent is not useful even when $\delta\approx\frac12$ as $\frac12-\delta+\frac14\cdot\frac{\delta(2\delta-1)}{(2\delta+1)(3-\delta)}<0$ but improvement of our exponent gives an essential spectral gap bound near $\frac{1}{2}$

\item[(4)] Our argument should extend to geometrically finite groups with \textit{parabolic} elements, where $\delta\ge\frac12$. Power Fourier decay was proved in this setting in \cite{LeclercPaukkonenSahlsten} using sum-product theory. Curiously, Fraser showed that sufficiently large parabolic rank bounds the Fourier dimension from above \cite{Fraser} and prevents these PS measures from being Salem.

\item[(5)] Our method also extends to Ahlfors-regular self-conformal measures satisfying a suitable non-linearity condition on the dual IFS, provided the dimension exceeds a threshold depending on the dual IFS. It also applies to perturbations of the systems considered here, but it is unclear whether it extends significantly further.
\end{itemize}
\end{rems}

\section{Preliminaries}
 \subsection{Hyperbolic plane}
 Let $\mathbb{H}^2 := \{z\in\C\,:\,\mathrm{Im}\,z > 0\}$ be the upper half-plane with its standard hyperbolic metric $\frac{1}{y^2}\left(\deriv x^2 + \deriv y^2\right)$. The group $\mathrm{PSL}_2(\R)$ acts isometrically on $\mathbb{H}^2 $ by Möbius transforms: for any \[\gamma = \left({\begin{array}{cc}
   a & b \\
   c & d \\
  \end{array} } \right),\] and $z\in \mathbb{H}^2$, we let \[\gamma(z) := \frac{az + b}{cz + d},\] and this action extends to the boundary $\partial\mathbb{H}^2 := \R\cup\{\infty\}$. Let us also write $\overline{\mathbb{H}^2} = \mathbb{H}^2 \cup \partial\mathbb{H}^2$.
\subsection{Schottky subgroups} \label{subsect:schottky}

In this subsection, we set up the notation on Schottky groups and coding, following \cite[\S2]{BD}, to which we refer for more details and references. Let $r \geq 2$ and fix non-intersecting closed hyperbolic half-disks $D_1, \ldots, D_{2r}$, centered on the real line. Let $\mathcal{A} := \{1, 2, \ldots, 2r\}$ and define for every $a \in \mathcal{A}$,
$\overline{a} := a - r$ if $a > r$, and $\overline{a} := a + r$ otherwise. Fix transformations $\gamma_a \in \mathrm{SL}_2(\R)$, $a \in \mathcal{A}$ such that for any $a \in \mathcal{A}$
\begin{equation}
\gamma_a\left(\overline{\mathbb{H}^2} \setminus D_{\overline{a}}^\circ\right) = D_a,
\label{eq:schottky}
\end{equation}
\[\gamma_{\overline{a}} = \gamma_a^{-1}.\]
(We have written $A^\circ$ for the interior of a subset $A$.) 

Let $\Gamma$ be the group generated by $\gamma_1, \ldots, \gamma_r$, which is free by the condition above and a ping-pong argument. Any convex-cocompact subgroup of $\mathrm{SL}_2(\R)$ is of this form, up to conjugation.

Moreover, up to conjugation, we can also make the following assumption:
\begin{assumption}
None of the disks $D_a$, $a \in \mathcal{A}$, contain $0 \in \R$.
\end{assumption}

\begin{rem}In what follows, we will frequently write $C = C(\Gamma)$ for a constant, possibly varying each time, which is allowed to depend on $\Gamma$. It will also be allowed to depend on the Schottky data $D_a, \gamma_a$ for $a = 1, \ldots, 2r$. Similarly, the notation $a \lesssim b$ will mean $a \leq Cb$ for some $C = C(\Gamma) > 0$ which depends only on $\Gamma$ and the Schottky data, and $a \asymp b$ means $a \lesssim b$ and $b\lesssim a$.
\end{rem}

\subsection{Action on the real line and limit set.}
Let $\mathbf{a} = a_1a_2\cdots a_n$ be a word on the alphabet $\mathcal{A}$. Call a word \emph{reduced} if for any $j = 1, \ldots, n - 1$, $a_j \neq \overline{a_{j+1}}$. Let $\mathcal{W}$ be the set of reduced words of any positive length, and for any $n\geq 1$ let $\mathcal{W}^n$ be the set of words of length $n$. For any word $\mathbf{a}$, we denote its length by $\abs{\mathbf{a}}$. Moreover, we write $\mathbf{a}' := a_1\cdots a_{n - 1}$. For any letter $b \in \mathcal{A}$, we write $\mathbf{a} \rightsquigarrow b$ if the last letter of $\mathbf{a}$ is $b$. Let us also write
\[\gamma_\mathbf{a} = \gamma_{a_1} \circ \gamma_{a_2} \circ \cdots \circ \gamma_{a_n},\]
and as a matrix, let $a_\mathbf{a}, b_\mathbf{a}, c_\mathbf{a}, d_\mathbf{a} \in \R$ be such that:
\[\gamma_{\mathbf{a}} = \left({\begin{array}{cc}
   a_\mathbf{a} & b_\mathbf{a} \\
   c_\mathbf{a} & d_\mathbf{a} \\
  \end{array} } \right).\]
Let us now look at the action on the boundary. For any $a\in\mathcal{A}$, let
\[I_a := D_a \cap \left(\R\cup\{\infty\}\right) \subset \R.\]
For any word of length at least $2$, $\mathbf{a} = a_1a_2\cdots a_n$, we define the cylinder set $I_\mathbf{a}$
\[I_\mathbf{a} = \gamma_{\mathbf{a}'}(I_{a_n}) = \gamma_{a_1a_2\cdots a_{n - 1}}(I_{a_n}).\]
Note that the intervals are nested: if $\mathbf{a}$ is a prefix of $\mathbf{b}$, then $I_\mathbf{b} \subset I_\mathbf{a}$. This is a consequence of equation \eqref{eq:schottky} and the fact that words are reduced. Moreover, if neither $\mathbf{a}$ nor $\mathbf{b}$ is a prefix of the other, then $I_\mathbf{a}$ and $I_\mathbf{b}$ are disjoint. Therefore, we define the \emph{limit set} as
\[\Lambda_\Gamma := \bigcap_{n \geq 1} \bigcup_{\mathbf{a} \in \mathcal{W}^n} I_\mathbf{a}.\]

\subsection{The associated IFS and bounded distortion}In what follows, we use results of \cite{BD}, but see also Li-Naud-Pan \cite{LNP}, whose treatment is closer to our own and to the techniques in the literature on IFSs, using bounded distortion as in the lemma below.
First, we note that for any $\mathbf{a} \in \mathcal{W}$, we have
\[\gamma_\mathbf{a}(x) = \frac{a_\mathbf{a}x + b_\mathbf{a}}{c_\mathbf{a}x + d_\mathbf{a}},\;
\gamma'_\mathbf{a}(x) = \frac{1}{\left(c_\mathbf{a}x + d_\mathbf{a}\right)^2},\;
\frac{\gamma''_\mathbf{a}}{\gamma'_\mathbf{a}}(x) = -\frac{2c_\mathbf{a}}{c_\mathbf{a}x + d_\mathbf{a}}.\]
\begin{notation}
For any $b \in \mathcal{A}$, let $U_b := I_b + (-\eta_0, \eta_0)$, where $\eta_0 > 0$ is a small constant depending only on $\Gamma$, such that the for any distinct $b, c \in \mathcal{A}$, $U_b$ and $U_c$ have positive distance, which we can choose to depend only on the Schottky data.
\end{notation}
\begin{lem}[Bounded distortion]
 There is $C = C(\Gamma)> 0$ such that for any $b \in \mathcal{A}$, $\mathbf{a} \in \mathcal{W}$ with $\abs{\mathbf{a}} \geq 2$ and $\mathbf{a} \rightsquigarrow b$, and $x \in U_b$, we have:
 \[\abs{\frac{\gamma_{\mathbf{a}'}''(x)}{\gamma_{\mathbf{a}'}'(x)}} \leq C.\]
 \label{lem:bounded-distortion}
\end{lem}
\begin{proof}
This follows from \cite[Lemma 2.5]{BD}, but we can also deduce it directly. Writing $\mathbf{a}' = a_1\cdots a_n$, note that
\[-\frac{d_{\mathbf{a}'}}{c_{\mathbf{a}'}} = \gamma_{\mathbf{a}'}^{-1}(\infty) = \gamma_{\overline{a_n}\cdots\overline{a_1}}(\infty).\]
Then since $\infty \notin I_{\overline{a_1}}$, we have $\gamma_{\overline{a_n}\cdots\overline{a_1}}(\infty) \in I_{\overline{a_n}}$.
Therefore, we have
\begin{align*}
\frac{\gamma_{\mathbf{a}'}''(x)}{\gamma_{\mathbf{a}'}'(x)} 
= \frac{-2c_{\mathbf{a}'}}{c_{\mathbf{a}'}x + d_{\mathbf{a}'}}
= \frac{-2}{x  + \frac{d_{\mathbf{a}'}}{c_{\mathbf{a}'}}}
\end{align*}and $-\frac{d_{\mathbf{a}'}}{c_{\mathbf{a}'}} \in I_{\overline{a_{n}}}$ while $x \in U_b$, and $b \neq \overline{a_{n}}$, so the denominator is bounded from below by \[\inf\{\abs{x - y}, x \in U_b, y \in I_c, b, c \in \mathcal{A}, b\neq c\},\]which is positive and depends only on the Schottky data.
\end{proof}
For any bounded interval $I$, denote by $\abs{I}$ its size. It is proved in \cite{BD} that there is a constant $\kappa = \kappa(\Gamma) \in (0,1)$ such that for any words $\mathbf{a},\mathbf{b} \in\mathcal{W}$ such that $\mathbf{a}$ is a prefix of $\mathbf{b}$, we have
\begin{equation}\label{eq:contraction}\abs{I_\mathbf{b}} \leq \kappa^{\abs{\mathbf{b}}-\abs{\mathbf{a}}}\abs{I_\mathbf{a}}.\end{equation}
Therefore, the set $\Lambda_\Gamma$ is a hyperbolic Cantor set.

Moreover, by bounded distortion (lemma \ref{lem:bounded-distortion}) and the mean-value theorem, we get the following:
\begin{lem}\label{lem:contraction}
There exists $C = C(\Gamma) > 0$ such that for any $b \in \mathcal{A}$, $\mathbf{a} \in \mathcal{W}$ with $\mathbf{a} \rightsquigarrow b$ and $x \in U_b$, we have
\[C^{-1}\abs{I_\mathbf{a}} \leq \abs{\gamma'_{\mathbf{a}'}(x)} \leq C\abs{I_\mathbf{a}}.\]
Therefore we deduce that $\gamma'_\mathbf{a}$ decreases exponentially with $\abs{\mathbf{a}}$ on $U_b$.
\end{lem}

Let us also record the following consequence:
\begin{lem}[Quasi-Bernoulli property]
\label{lem:quasi-bernoulli}
There is a constant $C = C(\Gamma) > 0$ such that for every two words $\mathbf{a}, \mathbf{b} \in \mathcal{W}$ such that $\mathbf{a}\mathbf{b} \in \mathcal{W}$, we have:
\[C^{-1}\abs{I_\mathbf{a}}\abs{I_\mathbf{b}} \leq \abs{I_{\mathbf{a}\mathbf{b}}} \leq C\abs{I_\mathbf{a}}\abs{I_\mathbf{b}}.\]
In particular, there is $C = C(\Gamma) > 0$ such that for every $\mathbf{a} \in \mathcal{W}$ with $\mathbf{a}'$ non-empty:
\[C^{-1}\abs{I_\mathbf{a}} \leq \abs{I_{\mathbf{a}'}} \leq C\abs{I_\mathbf{a}}.\]
\end{lem}
\begin{proof}
By definition, we have $I_{\mathbf{a}\mathbf{b}} = \gamma_\mathbf{a}(I_\mathbf{b})$, so by the mean value theorem, there is $x\in I_\mathbf{b}$ such that $\abs{I_{\mathbf{a}\mathbf{b}}} = \abs{\gamma_\mathbf{a}'(x)} \cdot \abs{I_\mathbf{b}}$. But because $x \in I_b$, where $b$ is the first letter of $\mathbf{b}$, and $\mathbf{a} \rightsquigarrow b$, we have $\abs{\gamma_\mathbf{a}'(x)} \asymp \abs{I_\mathbf{a}}$. This proves the result.
\end{proof}
\subsection{Patterson-Sullivan measure}
We let $\delta \in (0,1)$ be the critical exponent of the Poincaré series of $\Gamma$ and $\mu$ the Patterson-Sullivan measure for $\Gamma$ (centered at $i \in \mathbb{H}^2$). The measure $\mu$ is a measure on $\R$ with support $\Lambda_\Gamma$ and which satisfies, for every bounded Borel function $f$ and $\gamma \in \Gamma$:
\begin{equation}\int_{\Lambda_\Gamma} f\deriv \mu = \int_{\Lambda_\Gamma} f(\gamma(x))\abs{\gamma'(x)}_{\mathbb{S}^1}^\delta\deriv\mu(x),
\label{eq:equivariance}
\end{equation}
where \[\abs{\gamma'(x)}_{\mathbb{S}^1} = \frac{1 + x^2}{1 + \gamma(x)^2}\abs{\gamma'(x)}.\]

The measure $\mu$ is Ahlfors regular of exponent $\delta$ (for a proof see \cite[lemma 2.11 and 2.12]{BD}):
\begin{lem}[Ahlfors regularity of the Patterson-Sullivan measure]
\label{lem:ahlfors-regularity}
There is a constant $C = C(\Gamma) >0$ such that for any interval $J$ of radius $r > 0$, we have
\[\mu(J) \leq Cr^\delta,\]
and moreover if $J$ is centered at a point in $\Lambda_\Gamma$, we have
\[\mu(J) \geq C^{-1}r^\delta,\]
as well as
\[C^{-1}\abs{I_\mathbf{a}}^\delta \leq \mu(I_\mathbf{a}) \leq C\abs{I_\mathbf{a}}^\delta.\]
\end{lem}

\subsection{Stopping time}
The equivariance property \eqref{eq:equivariance} implies that for any bounded Borel function $f$, we have for any $b \in \mathcal{A}$, and $\mathbf{a}$ with $\mathbf{a} \rightsquigarrow b$:
\begin{equation}
\int_{I_\mathbf{a}} f(x)\deriv\mu(x)
= \int_{I_b} f(\gamma_{\mathbf{a}'}(x))w_{\mathbf{a}'}(x)\deriv \mu(x),
\label{eq:equivariance2}
\end{equation}
where we write
\[w_{\mathbf{a}'}(x) := \abs{\gamma_{\mathbf{a}'}'(x)}_{\mathbb{S}^1}^\delta.\]

Let us say that a set of words $Z$ is a \emph{partition} if any word in $\mathcal{W}$ of large enough length has a unique prefix in $Z$; in other words, $Z$ is a partition if
\[\Lambda_\Gamma = \bigsqcup_{\mathbf{a} \in Z} I_{\mathbf{a}},\]
the union being disjoint. We think of partitions as formalising the concept of stopping times: for a partition $Z$, formula \eqref{eq:equivariance2} above implies that for any bounded Borel function $f$, we have
\[\int f(x)\deriv\mu(x) = \sum_{b \in \mathcal{A}} \int_{I_b} \sum_{\mathbf{a} \in Z \atop \mathbf{a} \rightsquigarrow b} f(\gamma_{\mathbf{a}'}(x))w_{\mathbf{a}'}(x)\deriv \mu(x).\]

\begin{defn}
For any $\tau > 0$, define $\mathcal{W}_\tau$ to be the set of non-empty words such that $\abs{I_\mathbf{a}} \leq \tau$; and no prefix of $\mathbf{a}$ satisfies this property. (Here we set $\abs{I_{\varnothing}} = \infty$.)
\end{defn}
\begin{lem}
There is $C = C(\Gamma) > 0$ such that for any $\tau > 0$ small enough, the set $\mathcal{W}_\tau$ is a partition; we have 
$C^{-1}\tau \leq \abs{I_\mathbf{a}} \leq C\tau$, and
\[\#\mathcal{W}_\tau \leq C\tau^{-\delta},\]
and for any $b \in \mathcal{A}$, $x \in U_b$, and $\mathbf{a} \in W_\tau$ with $\mathbf{a} \rightsquigarrow b$, we have
\[C^{-1}\tau \leq \abs{\gamma'_{\mathbf{a}'}(x)} \leq C\tau,\]
and therefore also $C^{-1}\tau^{-1/2} \leq \abs{c_{\mathbf{a}'}x + d_{\mathbf{a}'}} \leq C\tau^{-1/2}$.
\end{lem}
\begin{proof}
By definition, $\abs{I_\mathbf{a}} \leq \tau$ and $\abs{I_{\mathbf{a}'}} > \tau$, but by lemma \ref{lem:quasi-bernoulli}, $\abs{I_\mathbf{a}} \asymp \abs{I_{\mathbf{a}'}}$; therefore $\abs{I_\mathbf{a}} \asymp \tau$. Therefore by lemma \ref{lem:ahlfors-regularity}, $\mu(I_\mathbf{a}) \asymp \tau^\delta$. 
Since
\[1 = \sum_{\mathbf{a} \in \mathcal{W}_\tau} \mu(I_\mathbf{a}),\]
we deduce that in fact $\#\mathcal{W}_\tau \asymp \tau^{-\delta}$. The other properties are a consequence of bounded distortion as above.
\end{proof}
\subsection{Some estimates on coefficients.} Here we record a number of simple technical estimates that will be needed below. The proofs are variants of the technique used in the proof of lemma \ref{lem:bounded-distortion}.
\begin{lem}\label{lem:bound-quotient}
There is $C = C(\Gamma) > 0$ such that for any $\mathbf{a} \in\mathcal{W}$, we have:
\[C^{-1} \leq \abs{\frac{c_\mathbf{a}}{d_\mathbf{a}}} \leq C.\]
\end{lem}
\begin{proof}
As above, $\gamma_{\overline{a_n}\cdots\overline{a_1}}(\infty) \in I_{\overline{a_n}}$. This proves the result, since $I_{\overline{a_n}}$ is a finite interval which has positive distance from $0$.
\end{proof}
\begin{lem}
\label{lem:bound-c-d}
 There is $C = C(\Gamma) > 0$ such that for any $\tau \in(0,1)$ and $\mathbf{a} \in\mathcal{W}_\tau$, we have:
 \[C^{-1} \tau^{-1/2} \leq \abs{c_\mathbf{a'}}, \abs{d_\mathbf{a'}} \leq C\tau^{-1/2}.\]
\end{lem}
\begin{proof}
%Write $\mathbf{a} = a_1\cdots a_{n- 1}a_n$. 
Let us start with the upper bound on $\abs{c_{\mathbf{a}'}}$. Choose $x, y\in I_{a_n}$ with $\abs{x - y} \geq C^{-1}$ for some $C = C(\Gamma) > 0$. Then we have
\[\abs{c_{\mathbf{a}'}x + d_{\mathbf{a}'}}, \abs{c_{\mathbf{a}'}y +d_{\mathbf{a}'}} \lesssim \tau^{-1/2},\]
so $\abs{c_{\mathbf{a}'}(x - y)} \lesssim \tau^{-1/2}$ and thus $\abs{c_{\mathbf{a}'}} \lesssim \tau^{-1/2}$.

For the lower bound, let $x \in I_{a_n}$. Using lemma \ref{lem:bound-quotient}, we have
\[\tau^{-1/2} \lesssim \abs{c_{\mathbf{a}'} x + d_{\mathbf{a}'}} = \abs{c_{\mathbf{a}'}} \cdot \abs{x + \frac{d_{\mathbf{a}'}}{c_{\mathbf{a}'}}} \lesssim \abs{c_{\mathbf{a}'}}.\]
For $d_{\mathbf{a}'}$, we can write, using lemma \ref{lem:bound-quotient}:
\[\abs{d_{\mathbf{a}'}} = \abs{\frac{d_{\mathbf{a}'}}{c_{\mathbf{a}'}}}\cdot\abs{c_{\mathbf{a}'}} \asymp \tau^{-1/2}.\]
\end{proof}

\subsection{Transpose of $\Gamma$ and the dual IFS}\label{sec:dualIFS}
An essential input for our proof is a so-called ``non-concentration bound'' for the ratios $c_\mathbf{a}/d_\mathbf{a}$, $\mathbf{a} \in \mathcal{W}_\tau$. We could very well use the proof given by Bourgain-Dyatlov \cite[proof of Lemma 2.16]{BD}, but we will reformulate it in a different way. Briefly speaking, the technique they use relies on the fact, which we have also used above, that
\[\gamma_\mathbf{a}^{-1}(\infty) = -\frac{d_\mathbf{a}}{c_\mathbf{a}}.\]
Rather than using the inverse image of infinity, we will consider below the image of zero by the transpose matrix. This is admittedly a very small change, but we believe that it makes the proof slightly clearer (though somewhat longer) and relates it to other works in the literature.

So let us consider the \textit{transpose group} of $\Gamma$:
\[\Gamma^T := \{\gamma^T\,:\,\gamma\in\Gamma\},\]
where $\gamma^T$ is the transpose of $\gamma$ in the group $\mathrm{PSL}_2(\R)$. Writing \[J := \left({\begin{array}{cc}
   0 & 1 \\
   -1 & 0 \\
  \end{array} } \right),\] we have $\gamma^T = J\gamma^{-1}J^{-1}$. Thus $\Gamma^T$ is still a Schottky group, and since it is conjugate to $\Gamma$, it is again convex co-compact Schottky with the same critical exponent $\delta$. Therefore its Patterson-Sullivan measure $\nu$ is Ahlfors $\delta$-regular. (In fact $\nu = J_{*}\mu$.)

\begin{lem}For any word $\mathbf{a} \in \mathcal{W}$ with $\mathbf{a} = a_1\ldots a_n$, write $\overline{\mathbf{a}} := \overline{a_n}\ldots\overline{a_1}$.

There is a constant $C = C(\Gamma) > 0$ such that for any $\mathbf{a} \in \mathcal{W}$, we have
\[C^{-1}\abs{I_\mathbf{a}} \leq \abs{I_{\overline{\mathbf{a}}}} \leq C\abs{I_\mathbf{a}} .\]
\end{lem}
\begin{proof}
Now we note that we have the following, for any $\mathbf{a} \in \mathcal{W}$ which can be written $\mathbf{a} = a_1\cdots a_n$. Then:
\begin{align*}
\abs{\gamma'_\mathbf{a}(0)} = \abs{\gamma'_{\mathbf{a}'}(\gamma_{a_n}(0))}\cdot\abs{\gamma'_{a_n}(0)} \asymp \abs{\gamma'_{\mathbf{a}'}(\gamma_{a_n}(0))} \asymp \abs{I_{\mathbf{a}}},
\end{align*}
where the last relation comes from the fact that $\gamma_{a_n}(0) \in I_{a_n}$ and lemma \ref{lem:contraction}. On the other hand, using the matrix representation it is clear that
\[(\gamma_\mathbf{a}^T)'(0) = \gamma_\mathbf{a}'(0).\]
But we also have:
\begin{align*}
\abs{(\gamma^T_\mathbf{a})'(0)} 
&= \abs{(\gamma^T_{a_n} \circ \cdots \circ \gamma^T_{a_2})'(\gamma^T_{a_1}(0))}\cdot\abs{(\gamma^T_{a_1})'(0)}\\
&\asymp \abs{(J \circ \gamma_{\overline{a_n}} \circ \cdots \circ \gamma_{\overline{a_2}} \circ J^{-1})'(J\gamma_{\overline{a_1}}(\infty))} \\
&\asymp \abs{\gamma'_{\overline{a_n}\cdots\overline{a_2}}(\gamma_{\overline{a_1}}(\infty))}\\
&\asymp \abs{I_{\overline{\mathbf{a}}}}.
\end{align*}
Here we used the assumption we made that the $I_a, a\in\mathcal{A}$ contain neither $0$ nor $\infty$, which guarantees that all the derivatives of $J$ (and $J^{-1} = J$) when using the chain rule are $\asymp 1$. For the last relation, we used lemma \ref{lem:contraction} and the fact that $\gamma_{\overline{a_1}}(\infty)) \in I_{\overline{a_1}}$. Combining all these estimates, we get the result.
\end{proof}

\begin{lem}\label{lem:NC-quotient}
Let $M > 0, \tau > 0$ and consider a set $W$ such that for every $\mathbf{a} \in W$, we have
\[M^{-1}\tau \leq \abs{I_\mathbf{a}} \leq M\tau.\]
There is a constant $C = C(\Gamma, M) > 0$ such that for every $y \in \R$ and $\sigma \geq \tau$, we have
\[\#\left\{\mathbf{a} \in W\,:\,\abs{\frac{c_\mathbf{a}}{d_\mathbf{a}} - y} \leq \sigma\right\} \leq C\tau^{-\delta}\sigma^{\delta}.\]
\end{lem}

\begin{proof}
The key to the proof is that for any $\mathbf{a} \in \mathcal{W}$, we have
\[
\gamma_\mathbf{a}^T(0)=\frac{c_\mathbf{a}}{d_\mathbf{a}}.
\]
Let
\[
\mathcal A=\left\{\mathbf{a}\in\mathcal W:\left|\frac{c_\mathbf{a}}{d_\mathbf{a}}-y\right|\le\sigma\right\}.
\]
Then for each $\mathbf a\in\mathcal A$, writing
$\mathbf{a} = a_1\cdots a_n$, we have
\begin{align*}
\frac{c_\mathbf a}{d_\mathbf a}
&=\gamma_\mathbf a^T(0)\\
&=\gamma_{a_n}^T\circ \cdots \circ \gamma_{a_1}^T(0)\\
&=J \circ \gamma_{a_n}^{-1}\circ \cdots \circ \gamma_{a_1}^{-1}\circ J^{-1}(0)\\
&=J\circ\gamma_{\overline{a_n}\cdots\overline{a_1}}(\infty)\\
&\in J(I_{\overline{\mathbf{a}}}) =: J_{\overline{\mathbf{a}}},
\end{align*}
where we have used that $\infty\notin I_{a_1}$. Thus
$J_{\overline{\mathbf a}}$ meets $B(y,\sigma)$ non-trivially. Since $J$ is bilipischitz,
$|J_{\overline{\mathbf a}}|\asymp_M\tau$. Because $\sigma\ge\tau$, there is
$C=C(\Gamma, M)>0$ such that
\[
J_{\overline{\mathbf a}}\subset B(y,C\sigma).
\]

Moreover, the cylinders $J_{\overline{\mathbf a}}$ have uniformly bounded
overlap, depending only on $M$ and $\Gamma$, by lemma \ref{lem:bounded-multiplicity} below, and satisfy
\[
\nu(J_{\overline{\mathbf a}})
\asymp |J_{\overline{\mathbf a}}|^\delta
\asymp_M \tau^\delta,
\]
by lemma \ref{lem:ahlfors-regularity} and the fact that $\nu = J_{*}\mu$. Hence
\[
\#\mathcal A\,\tau^\delta
\lesssim_M
\sum_{\mathbf a\in\mathcal A}\nu(J_{\overline{\mathbf a}})
\lesssim
\nu(B(y,C\sigma))
\lesssim_M
\sigma^\delta,
\]
where the last inequality is Ahlfors regularity of $\nu$. Thus
\[
\#\mathcal A\lesssim_M \tau^{-\delta}\sigma^\delta,
\]
as claimed.
\end{proof}
\begin{lem}
\label{lem:bounded-multiplicity}
Let $M > 1$ be a constant. Let $\tau > 0$ and consider a set of words $B \subset\mathcal{W}$ such that for every $\mathbf{b} \in B$, we have
$M^{-1}\tau \leq \abs{I_\mathbf{b}} \leq M\tau.$
Then there is a constant $K = K(\Gamma, M)$ which depends only on $\Gamma$ and $M$ such that the covering $\{I_\mathbf{b}\}_{\mathbf{b} \in B}$ has multiplicity bounded by $K$, that is,
$\sum_{\mathbf{b} \in B} \mathbf{1}_{I_\mathbf{b}} \leq K.$
\end{lem}
\begin{proof}
The intervals $I_\mathbf{b}$ are nested, that is, if $I_\mathbf{b}$ and $I_\mathbf{c}$ meet, either $\mathbf{b}$ is a prefix of $\mathbf{c}$, or $\mathbf{c}$ is a prefix of $\mathbf{b}$. Assume first that $\mathbf{c}$ is a prefix of $\mathbf{b}$. Then we have
\[M^{-1}\tau \leq \abs{I_{\mathbf{b}}} \leq \kappa^{\abs{\mathbf{b}}-\abs{\mathbf{c}}}\abs{I_\mathbf{c}} \leq \kappa^{\abs{\mathbf{b}}-\abs{\mathbf{c}}}M\tau,\]
where we have used uniform contraction (equation \eqref{eq:contraction}), and therefore $\abs{\mathbf{b}} - \abs{\mathbf{c}}$ is less than a constant depending only on $\Gamma$ and $M$. Therefore, there are finitely many $\mathbf{b}$ such that $I_\mathbf{b}$ is included in $I_\mathbf{c}$. If instead $\mathbf{b}$ is a prefix of $\mathbf{c}$, we can exchange the roles of $\mathbf{c}$ and $\mathbf{d}$ to deduce that there are finitely many possibilities again, with a bound depending only on $\Gamma$ and $M$. This proves the result.
\end{proof}

\begin{rem}
 As we said, this proof is present, in a shorter form avoiding the transpose, in \cite{BD}. It also appears for the special case of the Gauss map in notes of Naud \cite{Naud} where it is used to check the so-called ``uniform non-integrability" condition; in that case the Fuchsian group is generated by symmetric matrices so there is no need for the transpose. Finally, this also fits well with the point of view of the ``dual IFS" introduced in \cite{BKT}. There, a notion of dual IFS is defined as an infinite dimensional system. In our case it would be defined as follows: let $\Phi^* := \{F_\gamma : \gamma \in \Gamma\}$, where
$$F_\gamma(h)(x)
:=
\gamma'(x)\, h(\gamma x)
+
\frac{\gamma''(x)}{\gamma'(x)} , \quad x \in \mathbb{R} \cup \{\infty\}.
$$
The one-dimensional manifold $M = \{x \mapsto  h_t(x):=-\frac{2t}{tx+1}:t\in\mathbb R\cup\{\infty\} \}$ is invariant under the dual action: $F_\gamma(h_t)=h_{\gamma^T(t)}$ for the transpose $\gamma^T$. Indeed, $h_{\tilde{\gamma}^T(0)}(x) = \frac{\tilde{\gamma}''(x)}{\tilde{\gamma}'(x)}$ for any $\tilde{\gamma} \in \mathrm{SL}_2(\R)$. Then the formula follows by two applications of the chain rule. Thus the dual attractor $\Lambda_\Gamma^*$ is the set of accumulation points of $\{F_{\mathbf{a}}(0) : \mathbf{a} \in \mathcal{W}\}$, which is equal to $\{h_t : t \in \Lambda_{\Gamma^T}\}$, where $\Lambda_{\Gamma^T}$ is the limit set of $\Gamma^T$. This follows from the fact that $\Lambda_{\Gamma^T}$ is the closure of the set of $\Gamma^T(0) \in \R \cup \{\infty\}$, for $\gamma\in\Gamma$.
\end{rem}

\section{Proof}

\subsection{Reduction to oscillatory Lebesgue integrals}

The proof is fundamentally based on the following approximation lemma:
\begin{lem}\label{lma:approximation}
Let $I$ be an interval. Let $m$ be a probability measure on $I$, satisfying a Frostman condition with exponent $s$ and constant $1$: for every interval $J$ of radius $r > 0$, we have $m(J) \leq r^s.$ 

Let $f$ be a bounded $C^1$ function on $I$. Then for every $\eta > 0$ and every interval $J$ with $J + [-\eta, \eta] \subset I$, we have:
\[\int_J \abs{f}\deriv m \leq \eta^{(s - 1)/2}\norm{f}_{L^2(J + [-\eta,\eta], \Leb)} + \eta \norm{f'}_{\infty}.\]
\label{lem:approximation-measure}
\end{lem}

\begin{proof}
Let $\alpha > 0$ to be chosen later. Let $P_\eta = \frac{1}{2\eta}\mathbf{1}_{[-\eta, \eta]}$ and $m_\eta = m * P_\eta$. Then $m_\eta$ is absolutely continuous with respect to Lebesgue measure and the density is 
$\frac{\deriv m_\eta(x)}{\deriv x} = \frac{1}{2\eta}m(B(x, \eta))\leq\frac{1}{2}\eta^{s-1}.$ Then we can write
\begin{align*}
\int_J \abs{f}\deriv m
&= \int_J \abs{f}*P_\eta \deriv m + \int_J (\abs{f} - \abs{f}*P_\eta)\deriv m \\
&\leq \int_{J + [-\eta,\eta]}\abs{f}\deriv m_\eta +\eta\norm{f'}_\infty \\
&= \int_{J + [-\eta,\eta]} \abs{f}\mathbf{1}_{\{\abs{f} > \alpha\}} \deriv m_\eta + \int_{J + [-\eta,\eta]} \abs{f}\mathbf{1}_{\abs{f} \leq \alpha} \deriv m_\eta +  \eta\norm{f'}_\infty \\
&\leq \frac{1}{\alpha}\int_J \abs{f}^2 \frac{\deriv m_\eta(x)}{\deriv x}\deriv x + \alpha + \eta\norm{f'}_\infty \\
&\leq \frac{\eta^{s-1}}{\alpha}\norm{f}_{L^2(J + [-\eta,\eta], \Leb)}^2 + \alpha + \eta\norm{f'}_\infty.
\end{align*}
Now we choose $\alpha = \eta^{(s - 1)/2}\norm{f}_{L^2(J + [-\eta,\eta], \Leb)}$. This gives the required upper bound.
\end{proof}

Let us now fix $\xi \in \R$ with $\abs{\xi}$ large enough, depending only on $\Gamma$, and study a function defined for any $b \in \mathcal{A}$ and $x \in U_b$ by
\begin{equation}\label{eq:F-def}
f(x) = f_\xi(x) 
:=
\sum_{\mathbf{a}\in\mathscr{W}_\tau \atop \mathbf{a} \rightsquigarrow b}
e^{-2\pi i\xi\gamma_{\mathbf{a}'}(x)}\,w_{\mathbf{a}'}(x).
\end{equation}
Here we recall that $w_{\mathbf{a}'}(x) = \abs{\gamma_{\mathbf{a}'}(x)}_{\mathbb{S}^1}^\delta.$
Then by the properties of the Patterson-Sullivan measure, we have:
$\widehat{\mu}(\xi)=\int f\,d\mu.$
We also have the following bound:
\begin{lem}
\label{lem:bound-derivative-f}
    There is a constant $C = C(\Gamma) > 0$ such that
    \[\norm{f'(x)}_\infty \leq C(1 + \tau\abs{\xi}).\]
\end{lem}
This is in turn a consequence of the following simple lemma, which follows from the definition of $\mathcal{W}_\tau$ and bounded distortion (lemma \ref{lem:bounded-distortion}):
\begin{lem}
    \label{lem:bound-w}
    There is a constant $C = C(\Gamma) > 0$ such that for every $b \in \mathcal{A}$, $\mathbf{a} \in \mathcal{W}_\tau$ with $\mathbf{a} \rightsquigarrow b$, we have:
    \[\sup_{x\in U_b}\abs{w_{\mathbf{a}'}(x)}, \sup_{x\in U_b}\abs{w'_{\mathbf{a}'}(x)} \leq C\tau^\delta.\]
\end{lem}
We also have the following $L^2$ bound, which we will prove in the next section, and is the main part of the proof:
\begin{lem}\label{lma:L2}
    There is a constant $C = C(\Gamma) > 0$ such that for any $b\in\mathcal{A}$,
    \[\norm{f}_{L^2(U_b, \Leb)}^2 \leq C\left( \abs{\xi}^{-1/2}\tau^{-1/2} + \abs{\xi}^{-1}\tau^{-1} +\tau^\delta\right).\]
\end{lem}
Assuming we have proved these, we can conclude the proof of Theorem \ref{thm:main}.

\begin{proof}[Proof of Theorem \ref{thm:main}]
We use lemmas \ref{lem:approximation-measure}, \ref{lem:bound-derivative-f}, \ref{lma:L2}. We just need to choose $\eta$ and $\tau$ appropriately.
First we choose $\tau$ to make the two extreme terms in the $L^2$ bound comparable: $\tau = \abs{\xi}^{-\frac{1}{2\delta + 1}}$, so that by lemma \ref{lma:L2} we have
\[\norm{f}_{L^2(U_b, \Leb)} \lesssim \abs{\xi}^{-\frac{\delta}{2(2\delta + 1)}},\]
Moreover, by Lemma \ref{lem:bound-derivative-f}:
\[\norm{f'(x)}_\infty \lesssim \abs{\xi}^{\frac{2\delta}{2\delta + 1}}.\]
We now choose $\eta$ to make the two terms in lemma \ref{lem:approximation-measure} comparable:
\[\eta := \abs{\xi}^{-\frac{5\delta}{(3 - \delta)(2\delta + 1)}}.\]
Then if $\abs{\xi} \gtrsim 1$, as we may assume, $I_b + [-\eta, \eta] \subset U_b$. Thus using Lemma \ref{lma:approximation} with the two bounds above:
\begin{align*}
\abs{\hat{\mu}(\xi)}
&= \abs{\int f\deriv\mu}\\
&\leq\sum_{b\in\mathcal{A}} \int_{I_b} \abs{f}\deriv \mu \\
&\lesssim \eta^{(\delta - 1)/2}\abs{\xi}^{-\frac{\delta}{2(2\delta + 1)}} + \eta\abs{\xi}^{\frac{2\delta}{2\delta + 1}}\\ 
&\lesssim \abs{\xi}^{-\frac{\delta(2\delta - 1)}{(2\delta + 1)(3 - \delta)}}.
\end{align*}
\end{proof}

\subsection{Proof of the $L^2(\Leb)$ bound}

We now prove Lemma \ref{lma:L2}. We use the following oscillatory integral bound, slightly modified from \cite[Lemma 6.4]{JS}, which follows by integration by parts, see \cite[Section 4]{Kaufman} and \cite[Lemma 6.4]{JS} for the proofs.
\begin{lem}
\label{lem:oscillatory-integral-bound}
Let $\psi$ be a $C^2$ function such that there exists a $C^1$ function $\varphi$ and $\alpha, \beta$ with
\[\psi'(x) = \varphi(x)\cdot(\alpha x + \beta).\]
Assume moreover that there is $M > 0$ and $r > 0$ such that for every $x$,
\[\frac{1}{M}r \leq \abs{\varphi(x)} \leq Mr,\]
\[\abs{\varphi'(x)} \leq Mr.\]
Let $I$ be a bounded interval included in $[-M, M]$. Then there is a constant $C = C(M) > 0$ such that the following holds for any $\lambda \neq 0$ and $a \in C^1(I)$ with norm $\|a\|_{C^1} := \|a\|_\infty+\|a'\|_\infty$:
\begin{enumerate}
\item If $\alpha \neq 0$, then
\[\abs{\int_I e^{2\pi i\lambda \psi(x)}a(x)\deriv x} \leq C\left(r\abs{\alpha\lambda}\right)^{-1/2}\norm{a}_{C^1}.\]
\item If $\beta \neq 0$ and $\abs{\alpha} \leq \frac{1}{2M}\abs{\beta}$, then
\[\abs{\int_I e^{2\pi i\lambda \psi(x)}a(x)\deriv x }\leq C\left(r\abs{\beta\lambda}\right)^{-1}\norm{a}_{C^1}.\]
\end{enumerate}
\end{lem}
Case (1) corresponds to the phase having a critical point, hence the order of decay being $\abs{\lambda}^{-1/2}$, and case (2) to the absence of critical points.
\begin{comment}
\begin{proof}
For any $C^1$ function $a$ and $C^2$ function $f$ with $\abs{f} \geq c > 0$, we have
\begin{align*}
    \abs{\int_0^1 e^{if(x)}a(x)\deriv x}
    &= \abs{\int_0^1 e^{if(x)}if'(x)\cdot\frac{a(x)}{if'(x)}\deriv x}\\
    &= \abs{\left[\frac{e^{if(x)}{if'(x)}}a(x)\right]^1_0- \int_0^1 e^{if(x)}\cdot\frac{a'(x)}{if'(x)}\deriv x - \int_0^1 e^{if(x)}\cdot\frac{a(x)f''(x)}{f'(x)^2}\deriv x}\\
    &\lesssim \left(c^{-1} +  c^{-2}\norm{f''}_\infty\right)\norm{a}_{C^1}
\end{align*}
and of course the bound remains valid, changing the implied constant, if the integration is over a different bounded interval.
\end{proof}
\end{comment}

Let us fix $b \in \mathcal{A}$. We now expand the $L^2(U_b, \Leb)$ norm of $f$:
\[
\norm{f}_{L^2(U_b, \Leb)}^2
=
\int_{U_b}\abs{f(x)}^2\,dx
=
\sum_{\mathbf{a},\mathbf{b}\in\mathscr{W}_\tau\atop \mathbf{a},\mathbf{b}\rightsquigarrow b}
\int_{U_b} e^{-2\pi i\xi(\gamma_{\mathbf{a}'}(x)-\gamma_{\mathbf{b}'}(x))}
a_{\mathbf{a},\mathbf{b}}(x)\,dx,
\]
where for $x \in U_b$:
\[
a_{\mathbf{a},\mathbf{b}}(x) = w_{\mathbf{a}'}(x)w_{\mathbf{b}'}(x).
\]
For any $\mathbf{a},\mathbf{b}\in\mathscr{W}_\tau$ with $\mathbf{a}, \mathbf{b} \rightsquigarrow b$, we define
\[\mathcal{I}(\mathbf{a},\mathbf{b}) := \int_{U_b} e^{-2\pi i\xi(\gamma_{\mathbf{a}'}(x)-\gamma_{\mathbf{b}'}(x))}a_{\mathbf{a},\mathbf{b}}(x)\deriv x.\]
Define also for $x \in U_b$:
\[\psi_{\mathbf{a},\mathbf{b}}(x) := \gamma_{\mathbf{a}'}(x)-\gamma_{\mathbf{b}'}(x),\]
so that
\begin{align*}
\psi'_{\mathbf{a},\mathbf{b}}(x)
&= \frac{1}{(c_{\mathbf{a}'}x + d_{\mathbf{a}'})^2} - \frac{1}{(c_{\mathbf{b}'}x + d_{\mathbf{b}'})^2} \\
&= -\frac{\left((c_{\mathbf{a}'}x + d_{\mathbf{a}'}) - (c_{\mathbf{b}'}x + d_{\mathbf{b}'})\right)\left((c_{\mathbf{a}'}x + d_{\mathbf{a}'}) + (c_{\mathbf{b}'}x + d_{\mathbf{b}'})\right)}{\abs{c_{\mathbf{a}'}x + d_{\mathbf{a}'}}^2\abs{c_{\mathbf{b}'}x + d_{\mathbf{b}'}}^2}
\end{align*}
Now consider those $x \in U_b$ such that $\abs{(c_{\mathbf{a}'}x + d_{\mathbf{a}'}) - (c_{\mathbf{b}'}x + d_{\mathbf{b}'})} \leq \frac{1}{K}\tau^{-1/2}$ for some large $K > 0$; they constitute a (possibly empty) interval $J^{-}(\mathbf{a}, \mathbf{b})$. (Here recall that $\abs{c_{\mathbf{a}'}x + d_{\mathbf{a}'}}, \abs{c_{\mathbf{b}'}x + d_{\mathbf{b}'}} \asymp \tau^{-1/2}$ by lemma \ref{lem:contraction}.) Then
\begin{align*}
\abs{(c_{\mathbf{a}'}x + d_{\mathbf{a}'}) + (c_{\mathbf{b}'}x + d_{\mathbf{b}'})}
&\geq 2\abs{c_{\mathbf{a}'}x + d_{\mathbf{a}'}} - \abs{(c_{\mathbf{a}'}x + d_{\mathbf{a}'}) - (c_{\mathbf{b}'}x + d_{\mathbf{b}'})}\\
%&\gtrsim \abs{c_\mathbf{a}x + d_\mathbf{a}}\\
&\gtrsim \tau^{-1/2}.
\end{align*}
if $K$ was chosen large enough. Moreover, the reverse inequality is also true, so we can define
\[\varphi^{-}_{\mathbf{a},\mathbf{b}}(x) := \frac{(c_{\mathbf{a}'}x + d_{\mathbf{a}'}) + (c_{\mathbf{b}'}x + d_{\mathbf{b}'})}{\abs{c_{\mathbf{a}'}x + d_{\mathbf{a}'}}^2\abs{c_{\mathbf{b}'}x + d_{\mathbf{b}'}}^2},\] so that
\[\frac{\deriv}{\deriv x}\psi^{-}_{\mathbf{a},\mathbf{b}}(x) = \varphi^{-}_{\mathbf{a},\mathbf{b}}(x)\cdot\left((c_{\mathbf{a}'}x + d_{\mathbf{a}'}) - (c_{\mathbf{b}'}x + d_{\mathbf{b}'})\right).\] Thus we have for such $x \in J^{-}(\mathbf{a}, \mathbf{b})$:
\[\abs{\varphi^{-}_{\mathbf{a},\mathbf{b}}(x)} \asymp \tau^{3/2}.\]
Moreover, since $\abs{c_{\mathbf{a}'}}, \abs{d_{\mathbf{a}'}} \lesssim \tau^{-1/2}$ by lemma \ref{lem:bound-c-d}, one can check that also
\[\abs{\frac{\deriv}{\deriv x}\varphi^{-}_{\mathbf{a},\mathbf{b}}(x)} \lesssim \tau^{3/2}.\]
Similarly, one can construct an interval $J_{\mathbf{a}, \mathbf{b}}^{+}$ such that for every $x \in J^+(\mathbf{a}, \mathbf{b})$, we have
\begin{align*}
\abs{\psi'_{\mathbf{a},\mathbf{b}}(x)}
&\asymp \tau^{3/2}\abs{(c_{\mathbf{a}'}x + d_{\mathbf{a}'}) + (c_{\mathbf{b}'}x + d_{\mathbf{b}'})}.
\end{align*}
and a similar construction of a function $\varphi^{+}_{\mathbf{a},\mathbf{b}}$ as above.
Moreover, if $K$ was chosen large enough, these two intervals $J^\pm(\mathbf{a},\mathbf{b})$ do not intersect. Therefore we let 
\[\mathcal{I}^\pm(\mathbf{a}, \mathbf{b}) := \int_{J^\pm(\mathbf{a},\mathbf{b})} e^{-2\pi i\xi\psi_{\mathbf{a},\mathbf{b}}(x)}a_{\mathbf{a},\mathbf{b}}(x)\deriv x,\]
\[\mathcal{I}^0(\mathbf{a}, \mathbf{b}) := \int_{J^0(\mathbf{a},\mathbf{b})} e^{-2\pi i\xi\psi_{\mathbf{a},\mathbf{b}}(x)}a_{\mathbf{a},\mathbf{b}}(x)\deriv x,\]
where $J^0(\mathbf{a},\mathbf{b}) := U_b\setminus \left(J^-(\mathbf{a},\mathbf{b}) \cup J^+(\mathbf{a},\mathbf{b})\right)$

We now focus on bounding $\mathcal{I}^-(\mathbf{a}, \mathbf{b})$; the bound for $\mathcal{I}^+(\mathbf{a}, \mathbf{b})$ follows by a similar argument, and that for $\mathcal{I}^0(\mathbf{a}, \mathbf{b})$ will follow by a slightly different, and simpler, argument. First we record the following consequence of lemma \ref{lem:bound-w}:
\begin{lem}
There is a constant $C = C(\Gamma) > 0$ such that for any $\mathbf{a}, \mathbf{b} \in \mathcal{W}_\tau$, we have
\[\sup_{x\in U_b} \abs{a_{\mathbf{a},\mathbf{b}}(x)}, \sup_{x\in U_b} \abs{a'_{\mathbf{a},\mathbf{b}}(x)} \leq C\tau^{2\delta}.\]
\end{lem}
Therefore we get, using lemma \ref{lem:oscillatory-integral-bound}:
\begin{lem}
\label{lem:bound-pieces}
There is a constant $C = C(\Gamma) > 0$ such that for any $\mathbf{a}, \mathbf{b} \in \mathcal{W}_\tau$, the following holds:
\begin{enumerate}
\item If $c_{\mathbf{a}'} - c_{\mathbf{b}'} \neq 0$, then for $\abs{\xi} \geq 1$,
\[\abs{\mathcal{I}^{-}(\mathbf{a}, \mathbf{b})} \leq \frac{C\tau^{-3/4}\abs{\xi}^{-1/2}}{\abs{c_{\mathbf{a}'} - c_{\mathbf{b}'}}^{1/2}}\tau^{2\delta}.\]
\item If $d_{\mathbf{a}'} - d_{\mathbf{b}'} \neq 0$ and $\abs{c_{\mathbf{a}'} - c_{\mathbf{b}'}} \leq \frac{1}{C}\abs{d_{\mathbf{a}'} - d_{\mathbf{b}'}}$, then for $\abs{\xi} \geq 1$,
\[\abs{\mathcal{I}^{-}(\mathbf{a}, \mathbf{b})} \leq \frac{C\tau^{-3/2}\abs{\xi}^{-1}}{\abs{d_{\mathbf{a}'} - d_{\mathbf{b}'}}}\tau^{2\delta}.\] 
\end{enumerate}
\end{lem}

Thus to bound the Fourier transform, we are left with studying how often the differences $\abs{c_{\mathbf{a}'} - c_{\mathbf{b}'}}$ or $\abs{d_{\mathbf{a}'} - d_{\mathbf{b}'}}$ are small. For this we can use the following non-concentration results which will be proven below: 
\begin{lem}[Non-concentration]\label{lem:nonconcentration}
There is a constant $C = C(\Gamma) > 0$ such that for every $\tau > 0$, $\sigma \geq \tau$ and $\mathbf{a} \in \mathcal{W}_\tau$, we have
\[\#\left\{\mathbf{b} \in \mathcal{W}_\tau\,:\, \abs{c_{\mathbf{a}'} - c_{\mathbf{b}'}} + \abs{d_{\mathbf{a}'} - d_{\mathbf{b}'}} \leq \tau^{-1/2}\sigma 
\text{ or } \abs{c_{\mathbf{a}'} + c_{\mathbf{b}'}} + \abs{d_{\mathbf{a}'} + d_{\mathbf{b}'}} \leq \tau^{-1/2}\sigma \right\} \leq C\tau^{-\delta}\sigma^\delta.\]
\end{lem}

Lemma \ref{lem:nonconcentration} is consequence of Lemma \ref{lem:NC-quotient} on non-concentration arising from $\delta$-regularity of the Patterson-Sullivan measure of the transpose group $\Gamma^T$.

\begin{proof}[Proof of Lemma \ref{lem:nonconcentration}]
Let $\mathbf{a}, \mathbf{b} \in \mathcal{W}_\tau$ and note that
\begin{align*}
\abs{\frac{c_{\mathbf{a}'}}{d_{\mathbf{a}'}} - \frac{c_{\mathbf{b}'}}{d_{\mathbf{b}'}}}
&= \abs{\frac{\left(c_{\mathbf{a}'} - c_{\mathbf{b}'}\right)d_{\mathbf{b}'} +c_{\mathbf{b}'}\left(d_{\mathbf{b}'} - d_{\mathbf{a}'}\right)}{d_{\mathbf{a}'}d_{\mathbf{b}'}}}\\
&\leq \frac{\abs{c_{\mathbf{a}'} - c_{\mathbf{b}'}}}{\abs{d_{\mathbf{a}'}}} + \abs{\frac{c_{\mathbf{b}'}}{d_{\mathbf{b}'}}}\frac{\abs{d_{\mathbf{b}'} - d_{\mathbf{a}'}}}{\abs{d_{\mathbf{a}'}}}\\
&\lesssim \tau^{1/2}\left(\abs{c_{\mathbf{a}'} - c_{\mathbf{b}'}} + \abs{d_{\mathbf{a}'} - d_{\mathbf{b}'}}\right),
\end{align*}
where we have used lemmas \ref{lem:bound-quotient} and \ref{lem:bound-c-d}. Because by lemma \ref{lem:quasi-bernoulli}, $\abs{I_{\mathbf{a}'}} \asymp \abs{I_\mathbf{a}} \asymp \tau$, we can then apply lemma \ref{lem:NC-quotient}.

The proof for sums instead of differences is almost identical.
\end{proof}

Let $J$ be the largest integer such that $2^{-J} \geq \tau$. For every $j = 1, \ldots, J - 1$, consider
\[\mathcal{A}_j(\mathbf{a}) := \left\{\mathbf{b} \in \mathcal{W}_\tau\,:\,\tau^{-1/2}2^{-(j + 1)} \leq \abs{c_{\mathbf{b}'} - c_{\mathbf{a}'}} < \tau^{-1/2}2^{-j}, \abs{c_{\mathbf{b}'} - c_{\mathbf{a}'}} > \frac{1}{C}\abs{d_{\mathbf{b}'} - d_{\mathbf{a}'}}\right\},\]
where $C$ is the constant in lemma \ref{lem:bound-pieces}. We have
\begin{align*}
\#\mathcal{A}_j(\mathbf{a})
&\leq \#\left\{\mathbf{b}\in\mathcal{W}_\tau\,:\,\abs{c_{\mathbf{b}'} - c_{\mathbf{a}'}} + \abs{d_{\mathbf{b}'} - d_{\mathbf{a}'}} \leq (C + 1)\tau^{-1/2}2^{-j}\right\}
\lesssim \tau^{-\delta}2^{-\delta j}
\end{align*}
by lemma \ref{lem:nonconcentration}. Let also
\[\mathcal{A}'(\mathbf{a}) := \left\{\mathbf{b} \in \mathcal{W}_\tau\,:\, \abs{c_{\mathbf{b}'} - c_{\mathbf{a}'}} < \tau^{-1/2}2^{-J}, \abs{c_{\mathbf{b}'} - c_{\mathbf{a}'}} > \frac{1}{C}\abs{d_{\mathbf{b}'} - d_{\mathbf{a}'}}\right\},\]
then
\begin{align*}
\#\mathcal{A}'(\mathbf{a})
&\leq \#\left\{\mathbf{b}\in\mathcal{W}_\tau\,:\,\abs{c_{\mathbf{b}'} - c_{\mathbf{a}'}} + \abs{d_{\mathbf{b}'} - d_{\mathbf{a}'}} \leq (C + 1)\tau^{-1/2}2^{-J}\right\}
\lesssim \tau^{-\delta}2^{-\delta J} \lesssim 1,
\end{align*}
and for every $\mathbf{b} \in \mathcal{A}'(\mathbf{a})$, we have the trivial bound
\[\mathcal{I}^{-}(\mathbf{a},\mathbf{b}) \lesssim \tau^{2\delta}.\]
Then using lemma \ref{lem:bound-pieces}
\begin{align*}
\sum_{\mathbf{b} \in \mathcal{W}_\tau} \abs{\mathcal{I}^{-}(\mathbf{a}, \mathbf{b})}\mathbf{1}_{\left\{\abs{c_{\mathbf{b}'} - c_{\mathbf{a}'}} < \tau^{-1/2}, \abs{c_{\mathbf{b}'} - c_{\mathbf{a}'}} > \frac{1}{C}\abs{d_{\mathbf{b}'} - d_{\mathbf{a}'}}\right\}  }
&= \sum_{j = 0}^{J - 1}\sum_{\mathbf{b} \in \mathcal{A}_j(\mathbf{a})} \abs{\mathcal{I}^{-}(\mathbf{a}, \mathbf{b})} + \sum_{\mathbf{b} \in \mathcal{A}'(\mathbf{a})} \abs{\mathcal{I}^{-}(\mathbf{a}, \mathbf{b})} \\
&\lesssim \sum_{j = 0}^{J - 1} \abs{\xi}^{-1/2}\tau^{-3/4}\left(\tau^{-1/2}2^{-j}\right)^{-1/2}\tau^{2\delta}\cdot\tau^{-\delta}2^{-\delta j} + \tau^{2\delta} \\
&\lesssim \abs{\xi}^{-1/2} \tau^{-1/2}\tau^{\delta} \sum_{j = 0}^{J - 1}2^{(1/2-\delta)j} + \tau^{2\delta} \\
&\lesssim \abs{\xi}^{-1/2}\tau^{-1/2}\tau^{\delta} + \tau^{2\delta}.
\end{align*}

Now define
\[\mathcal{B}_j(\mathbf{a}) := \left\{\mathbf{b} \in \mathcal{W}_\tau\,:\,\tau^{-1/2}2^{-(j + 1)} \leq \abs{d_{\mathbf{b}'} - d_{\mathbf{a}'}} < \tau^{-1/2}2^{-j}, \abs{c_{\mathbf{b}'} - c_{\mathbf{a}'}} \leq \frac{1}{C}\abs{d_{\mathbf{b}'} - d_{\mathbf{a}'}}\right\},\]
where $C$ is the constant in lemma \ref{lem:bound-pieces}. As above, we have $\#\mathcal{B}_j(\mathbf{a})
\lesssim \tau^{-\delta}2^{-\delta j}$. Let also
\[\mathcal{B}'(\mathbf{a}) := \left\{\mathbf{b} \in \mathcal{W}_\tau\,:\, \abs{d_{\mathbf{b}'} - d_{\mathbf{a}'}} < \tau^{-1/2}2^{-J}, \abs{c_{\mathbf{b}'} - c_{\mathbf{a}'}} \leq \frac{1}{C}\abs{d_{\mathbf{b}'} - d_{\mathbf{a}'}}\right\},\]
then $\#\mathcal{B}'(\mathbf{a}) \lesssim 1$. So we get
\begin{align*}
\sum_{\mathbf{b} \in \mathcal{W}_\tau} \abs{\mathcal{I}^{-}(\mathbf{a}, \mathbf{b})}\mathbf{1}_{\left\{\abs{d_{\mathbf{b}'} - d_{\mathbf{a}'}} < \tau^{-1/2}, \abs{c_{\mathbf{b}'} - c_{\mathbf{a}'}} \leq \frac{1}{C}\abs{d_{\mathbf{b}'} - d_{\mathbf{a}'}}\right\}  }
&= \sum_{j = 0}^{J - 1}\sum_{\mathbf{b} \in \mathcal{B}_j(\mathbf{a})} \abs{\mathcal{I}^{-}(\mathbf{a}, \mathbf{b})} + \sum_{\mathbf{b} \in \mathcal{B}'(\mathbf{a})} \abs{\mathcal{I}^{-}(\mathbf{a}, \mathbf{b})} \\
&\lesssim \sum_{j = 0}^{J - 1} \abs{\xi}^{-1}\tau^{-3/2}\left(\tau^{-1/2}2^{-j}\right)^{-1}\tau^{2\delta}\cdot\tau^{-\delta}2^{-\delta j} + \tau^{2\delta} \\
&\lesssim \abs{\xi}^{-1} \tau^{-1/2}\tau^{\delta} \sum_{j = 0}^{J - 1}2^{(1-\delta)j} + \tau^{2\delta} \\
&\lesssim \abs{\xi}^{-1} \tau^{-1/2}\tau^{\delta}2^{(1 - \delta)J} + \tau^{2\delta} \\
&\lesssim \abs{\xi}^{-1}\tau^{-3/2}\tau^{2\delta} + \tau^{2\delta}.
\end{align*}

On the other hand, we also have using lemma \ref{lem:bound-pieces}
\begin{align*}
\sum_{\mathbf{b} \in \mathcal{W}_\tau} \abs{\mathcal{I}^{-}(\mathbf{a}, \mathbf{b})}\mathbf{1}_{\left\{\abs{c_{\mathbf{b}'} - c_{\mathbf{a}'}} \geq \tau^{-1/2}\right\}}
&\lesssim \sum_{\mathbf{b} \in \mathcal{W}_\tau} \abs{\xi}^{-1/2}\tau^{-3/4}\tau^{1/2}\tau^{2\delta}\\
&\lesssim \abs{\xi}^{-1/2}\tau^{-1/2}\tau^{\delta}
\end{align*}
and similarly
\begin{align*}
\sum_{\mathbf{b} \in \mathcal{W}_\tau} \abs{\mathcal{I}^{-}(\mathbf{a}, \mathbf{b})}\mathbf{1}_{\left\{\abs{d_{\mathbf{b}'} - d_{\mathbf{a}'}} \geq \tau^{-1/2}\right\}}
&\lesssim \abs{\mathcal{W}_\tau} \abs{\xi}^{-1}\tau^{-3/2}\tau^{1/2}\tau^{2\delta}\\
&\lesssim \abs{\xi}^{-1}\tau^{-1/2}\tau^{\delta}.
\end{align*}
Thus we have, summing all the pieces, and recalling that $\abs{\xi} \geq 1$:
\begin{align*}
\sum_{\mathbf{b}\in\mathcal{W}_\tau} \abs{\mathcal{I}^{-}(\mathbf{a}, \mathbf{b})} \lesssim \abs{\xi}^{-1/2}\tau^{-1/2}\tau^{\delta} + \abs{\xi}^{-1}\tau^{-3/2}\tau^{2\delta} + \tau^{2\delta}.
\end{align*}
Finally, summing over $\mathbf{a} \in \mathcal{W}_\tau$ we get
\[
\sum_{\mathbf{a},\mathbf{b}\in\mathcal{W}_\tau} \abs{\mathcal{I}^{-}(\mathbf{a}, \mathbf{b})} \lesssim \abs{\xi}^{-1/2}\tau^{-1/2} + \abs{\xi}^{-1}\tau^{-3/2}\tau^{\delta} + \tau^{\delta}.
\]
Adapting lemma \ref{lem:bound-pieces}, replacing differences $c_\mathbf{b} - c_\mathbf{a}$, $d_\mathbf{b} - d_\mathbf{a}$ with sums, we get similarly:
\[
\sum_{\mathbf{a},\mathbf{b}\in\mathcal{W}_\tau} \abs{\mathcal{I}^{+}(\mathbf{a}, \mathbf{b})} \lesssim \abs{\xi}^{-1/2}\tau^{-1/2} + \abs{\xi}^{-1}\tau^{-3/2}\tau^{\delta} + \tau^{\delta}.
\]
Let us now bound the part with $\mathcal{I}^0$.
\begin{claim}We have
\[\sum_{\mathbf{a},\mathbf{b}\in\mathcal{W}_\tau} \abs{\mathcal{I}^0(\mathbf{a},\mathbf{b})} \lesssim \abs{\xi}^{-1}\tau^{-1}.\]
\end{claim}
\begin{proof}
Let $\mathbf{a},\mathbf{b}\in\mathcal{W}_\tau$. Using the discussion above constructing $\mathcal{I}^-$, say, it is straightforward to check that for all $x \in J^0(\mathbf{a},\mathbf{b})$, we have
\[\abs{\psi'_{\mathbf{a},\mathbf{b}}(x)} \asymp \tau,\]
and
\[\abs{\psi''_{\mathbf{a},\mathbf{b}}(x)} \lesssim \tau.\]
Note that $J^0(\mathbf{a},\mathbf{b})$ is a union of at most 3 intervals. Therefore by integration by parts as in the proof of lemma \ref{lem:oscillatory-integral-bound}, we get
\[\abs{\mathcal{I}^0(\mathbf{a},\mathbf{b})} \lesssim \abs{\xi}^{-1}\tau^{2\delta - 1},\]
and the result follows by summing.
\end{proof}
Combining all the bounds, we get (using $\delta > 1/2$):
\[
\norm{f}_{L^2(U_b, \Leb)}^2 \leq \sum_{\mathbf{a},\mathbf{b}\in\mathcal{W}_\tau} \abs{\mathcal{I}(\mathbf{a}, \mathbf{b})} \lesssim \abs{\xi}^{-1/2}\tau^{-1/2} + \abs{\xi}^{-1}\tau^{-1} + \tau^{\delta}.
\]

\section*{Acknowledgements} We thank Semyon Dyatlov, Thomas Jordan, Gaétan Leclerc, Alexandre Minetto and Tomas Persson for useful discussions. 

\section*{AI statement} AI tools were used for proofreading the paper.

\end{document}